\documentclass[a4paper, 12pt]{amsart}

\pdfoutput=1

\usepackage{amsmath,amssymb,amsthm,url}
\usepackage[utf8]{inputenc}
\usepackage[T1]{fontenc}
\usepackage{libertine}
\usepackage[libertine,cmintegrals,cmbraces]{newtxmath}
\usepackage[shortlabels]{enumitem}
\usepackage{mathtools}
\usepackage{xcolor}
\usepackage{tikz}
\usepackage{tikz-cd}
\usepackage{nicefrac}
\usepackage{dsfont}
\usepackage{todonotes}
\usepackage{a4wide}
\usepackage{xcolor}

\AtBeginDocument{%
   \def\MR#1{}
} 
\usepackage{euscript}
\usepackage[all]{xy}

\usepackage[pagebackref]{hyperref}
  
\hypersetup{%
  bookmarksnumbered=true,%%
  colorlinks=true,%
  linkcolor=blue,%
  citecolor=blue,%
  filecolor=blue,%
  menucolor=blue,%
  urlcolor=blue,%
  pdfnewwindow=true,%
  pdfstartview=FitBH}

\definecolor{indiguy}{RGB}{195,179,227}

\usepackage{xcolor}
\definecolor{seagreen}{RGB}{46,139,87}
\definecolor{maroon}{RGB}{128,0,0}
\definecolor{darkviolet}{RGB}{148,0,211}
\definecolor{twelve}{RGB}{100,100,170}
\definecolor{thirteen}{RGB}{100,150,50}
\definecolor{fourteen}{RGB}{200,0,0}
\definecolor{fifteen}{RGB}{0,200,0}
\definecolor{sixteen}{RGB}{0,0,200}
\definecolor{seventeen}{RGB}{200,0,200}
\definecolor{eighteen}{RGB}{0,200,200}

\newcommand{\bb}[1]{\mathbb{#1}}

\newcommand{\es}[1]{\EuScript{#1}}
\renewcommand{\sf}[1]{\mathsf{#1}}

\DeclareMathOperator{\colim}{\mathrm{colim}}

\newcommand{\s}{{\sf{Sp}}}
\DeclareMathOperator{\T}{\es{S}_\ast}

\newcommand{\exc}[1]{\mathsf{Exc}^{\leq #1}}

\newcommand{\Fun}{\sf{Fun}}

\DeclareMathOperator{\Map}{\mathsf{Map}}

\DeclareMathOperator{\id}{\mathrm{Id}}

\newcommand{\op}{\mathrm{op}}

  \newcommand{\adjunction}[4]{
\xymatrix{
#1:#2 \ar@<.5ex>[r] &
\ar@<.5ex>[l] #3:#4
}}

\newtheorem{thm}{Theorem}[subsection]
\newtheorem{prop}[thm]{Proposition}
\newtheorem{lem}[thm]{Lemma}
\newtheorem{cor}[thm]{Corollary}

\newtheorem*{thm*}{Theorem}

\theoremstyle{definition}
\newtheorem{definition}[thm]{Definition}

\newtheorem{rem}[thm]{Remark}

\newtheorem{xxthm}{Theorem}

\newtheorem{ques}{Question}

\begin{document}

\title{The Goodwillie calculus of polyhedral products}
\author{Guy Boyde}
\address[Boyde]{Mathematical Institute, Utrecht University, Budapestlaan 6, 3584 CD Utrecht, The Netherlands}
\email{g.boyde@uu.nl}
\author{Niall Taggart}
\address[Taggart]{Mathematical Institute, Utrecht University, Budapestlaan 6, 3584 CD Utrecht, The Netherlands}
\email{n.c.taggart@uu.nl}

\begin{abstract}
We describe the Goodwillie calculus of polyhedral products in the case that the fat wedge filtration on the associated real moment-angle complex is trivial. We do this by analysing the behaviour on calculus of the Denham-Suciu fibre sequence, the Iriye-Kishimoto decomposition of the polyhedral product constructed from a collection of pairs of cones and their bases, and the Hilton-Milnor decomposition. As a corollary we show that the Goodwillie calculus of these polyhedral products converges integrally and diverges in $v_h$-periodic homotopy unless the simplicial complex is a full simplex.
\end{abstract}

\maketitle

\setcounter{tocdepth}{1}
{\hypersetup{linkcolor=black} \tableofcontents}

\section{Introduction}

Let $K$ be a finite simplicial complex and $(\underline{X}, \underline{A})=\{(X_i, A_i)\}$ be a collection of pairs of spaces indexed by the vertices of $K$. The \emph{polyhedral product} $(\underline{X}, \underline{A})^K$ is a subspace of the cartesian product $\prod X_i$ constructed from the combinatorial information of $K$. Of particular interest is the case where each subspace $A_i$ is a point, as then the polyhedral products interpolate between the wedge $\bigvee X_i$ (when $K$ consists of discrete points) and the product $\prod X_i$ (when $K$ is a full simplex). These spaces are a natural testing ground for homotopy theory and have concrete relationships with a surprisingly varied collection of mathematical disciplines, for a survey see e.g.,~\cite{BBCHandbook}. 

Goodwillie calculus is a tool for analysing functors that arise in topology and is in many ways a categorification of classical differential calculus. It approximates a homotopical functor $F$ by a tower of simpler functors
\[
\cdots \longrightarrow P_nF \longrightarrow P_{n-1}F \longrightarrow P_{1}F \longrightarrow P_0 F,
\]
in which $P_nF$ is the \emph{universal $n$-excisive functor under $F$}, analogous to the $n$-th order Taylor approximation to a function. We say that the tower \emph{converges} if $F$ is the (homotopy) limit of its approximations, and think of this as `having a convergent Taylor series'.

The Goodwillie tower can be regarded as interpolating between the stable and unstable worlds by viewing the former as a `linearization' of the latter. For example, in the case that $F$ is the identity functor $\id: \T \to \T$ on pointed spaces, $P_1(\id)$ is the stable homotopy functor $\Omega^\infty \Sigma^{\infty}$, and more generally, linear endofunctors of pointed spaces are essentially homology theories. Applications of this philosophy have been successful in several settings, for a survey see e.g.,~\cite{AroneChingHandbook}. The Goodwillie tower of the identity functor on pointed spaces has been extensively studied, see e.g.,~\cite{Johnson, AroneMahowald, AroneDwyer}, and is well understood. It is often the case that practitioners of Goodwillie calculus treat the Goodwillie tower of the identity as a \emph{primitive object}, providing all descriptions of the Goodwillie tower of more complex homotopical objects in terms of the tower of the identity.

Brantner and Heuts~\cite{BrantnerHeuts} study the behaviour of Goodwillie calculus on wedges, building on work of Arone and Kankaanrinta~\cite{AroneKankaanrinta}. Since the wedge functor $\bigvee: \es{S}_\ast^{\times m} \to \T$ is colimit preserving, there is a canonical equivalence 
\begin{align}
P_n(\bigvee)\simeq P_n(\id) \circ \bigvee,
\end{align} 
which allowed Brantner and Heuts~\cite[Theorem 1.2]{BrantnerHeuts} to provide a Hilton-Milnor decomposition for the Goodwillie tower of wedges.

\begin{thm*}[Brantner-Heuts]
For pointed connected spaces $X_1, \dots, X_m$, there is an equivalence  
\[
\Omega P_n(\id)(\Sigma X_1 \vee \cdots \vee \Sigma X_m) \simeq \sideset{}{'}\prod_{\omega \in \bb{L}_m} \Omega P_{\lfloor \frac{n}{|\omega|} \rfloor}(\id)(\Sigma\omega(X_1, \dots, X_m)), 
\]
where $\bb{L}_m$ is a basis for the free Lie algebra on $m$ generators.
\end{thm*}

The wedge is an instance of a polyhedral product on a simplicial complex consisting of discrete points, i.e., Brantner and Heuts investigated the Goodwillie calculus of the functor
\[
((-), \underline{\ast})^{[m]} = \bigvee : \es{S}_\ast^{\times m} \longrightarrow \T,
\]
and so one can ask if similar decompositions exist for the Goodwillie towers of polyhedral products, 
\[
((-), \underline{\ast})^{K} : \es{S}_\ast^{\times m} \longrightarrow \T,
\]
on arbitrary simplicial complexes. At the other extreme, when $K = \Delta^{m-1}$, the polyhedral product becomes the cartesian product
\[
((-), \underline{\ast})^{\Delta^{m-1}} = \prod : \es{S}_\ast^{\times m} \longrightarrow \T,
\]
and in this case, the calculus is also well-understood since for formal reasons we have
\begin{align}
P_n(\prod) \simeq \prod P_n(\id) \circ \sf{pr}_i,
\end{align}
where $\sf{pr}_i : \es{S}_\ast^{\times m} \to \T$ is the projection onto the $i$th factor (for details see Lemma~\ref{lem: multi-variable approx to product}). The reader may like to compare (1) and (2) and note that for both the wedge and the product the calculus is reducible to that of the identity for formal reasons, but that the formulas look quite different.

In this article, we fill in some of the blanks between the wedge sum and cartesian product by providing a decomposition of the Goodwillie tower of the polyhedral product $((-), \underline{\ast})^{K}$ in those cases where a natural integral decomposition of a certain homotopy fibre is known (c.f Lemma \ref{lem: Denham Suciu sequence}). Such decompositions are known in a wide variety of cases, see e.g., \cite{BBCGFunctor, GrbicTheriault, IriyeKishimoto, BBCHandbook}, and we work in the generality of Iriye and Kishimoto~\cite{IriyeKishimoto2} who require that the so-called fat wedge filtration is trivial, see \S~\ref{sec: integral decomps}. This hypothesis holds for many families of simplicial complexes including, $k$-skeleta of standard simplices, shifted complexes,  Alexander duals of shellable complexes, and $1$-skeletons of flag complexes that are chordal graphs.

\subsection*{Product decompositions} Let $K$ be a fixed simplicial complex on vertex set $[m]$. There are two different approaches one could take to studying the Goodwillie calculus of polyhedral products. Firstly, one can treat the functor 
\[
((-), \underline{\ast})^{K}: \es{S}_\ast^{\times m} \longrightarrow \T, \  \underline{X} \longmapsto (\underline{X},\underline{\ast})^{K},
\]
as a \emph{multi-variable} functor, and applying multi-variable calculus, i.e., calculus in each variable separately. In this case, we get $\Vec{n}$-excisive approximations for any multi-index $\Vec{n} = (n_1, \dots, n_m)$, and our product decomposition takes the following form.

\begin{xxthm}\label{Thm: multivariable decomp}
If the fat wedge filtration on the real moment-angle complex of $K$ is trivial, then for pointed connected spaces $X_1, \dots, X_m$ there is an equivalence, 
\[
\Omega P_{\Vec{n}}((\Sigma(-), \underline{\ast})^K)(\underline{X}) \simeq \Omega \sideset{}{'}\prod_{\omega \in \bb{L}(\mathscr{I}_K)} P_\kappa(\id)(\Sigma \omega \alpha (\underline{X}))\times \Omega \prod_{i=1}^m P_{n_i}(\id)(\Sigma X_i),
\]
where 
\begin{itemize}
    \item we denote by $\bb{L}(\mathscr{I}_K)$ the ordered set of Lie words $\omega$ forming a basis for the free algebra on formal symbols $\{y_{(I,\Vec{k})}\}_{(I, \vec{k}) \in \mathscr{I}_K}$. We evaluate such a word on the collection of functors 
\[
\alpha = \{\alpha_{(I,\vec{k})} : \es{S}_\ast^{\times m} \longrightarrow \es{S}_\ast \}_{(I, \vec{k}) \in \mathscr{I}_K},
\]
given by
\[
\alpha_{(I,\vec{k})}( \underline{X}) = |K_I| \wedge \bigwedge_{i \in I} X_{i}^{ \wedge k_i},
\]
by letting the bracket act as a point-wise smash product of functors, and,
\item we denote by $\kappa = \min_i(\lfloor \frac{n_i}{a_i} \rfloor)$ where $a_i$ the smash power of $X_i$ appearing in the `word' $\omega \alpha$ when applied to $\underline{X} = (X_1, \dots, X_m)$.
\end{itemize}
\end{xxthm}

Corollary \ref{cor: weak product is product after Pn} gives that for all but finitely many factors, $\kappa$ is zero and hence the weak product in Theorem~\ref{Thm: multivariable decomp} is really a finite cartesian product, in contrast with the integral situation. It is possible in principle, though undesirable in practice, to list the finitely many factors for which $\kappa$ is nonzero, together with the corresponding smash powers of the $X_i$, and arrive at a loop decomposition of the calculus of $P_{\Vec{n}}((\Sigma(-), \underline{\ast})^K)$ as a finite product of relatively brutal single variable approximations to the identity, evaluated on spaces obtained by smashing together the input spaces $X_i$ with geometric realizations of subcomplexes of $K$ coming from Iriye and Kishimoto's decomposition.

The second approach is to apply single-variable calculus to the functor $((-), \underline{\ast})^K$, that is, to treat $\underline{X}$ as a \emph{single-variable}. In this case, we obtain an analogous decomposition (Theorem \ref{thm: single-variable calc of fibre, restricted indexing}), but only after composing with the diagonal, i.e., taking all of the spaces to be the same. For many of the applications of polyhedral products to areas like, right–angled Artin groups, right–angled Coxeter, and graph products this suffices, see e.g.,~\cite{BBCHandbook}.

The decompositions we obtain here are quite different from the decompositions one would obtain from evaluating the Goodwillie calculus of the identity functor on the polyhedral product. For instance, in the case of the polyhedral product $(C\underline{X}, \underline{X})^K$, Theorem~\ref{thm: calc decomposition cone on X} provides a decomposition of the calculus of the functor $(C(-), (-))^K$ under the hypothesis that the fat wedge filtration on the real moment-angle complex is trivial. This decomposition is vasty different from the decomposition of the Goodwillie calculus of the identity functor evaluated on $(C\underline{X}, \underline{X})^K$, see Proposition~\ref{prop: BH for cone on X} and Remark~\ref{rem: compare with id}.

\subsection*{Convergence}
It is well-known that the Goodwillie tower of the identity converges on simply connected spaces. It then follows from the Seifert-Van Kampen theorem and the fact that the wedge preserves colimits that the Goodwillie tower of a wedge converges on simply connected spaces. The same holds true for the tower of a product: we have seen that this splits as a product. 

Brantner and Heuts~\cite[Theorem 1.5]{BrantnerHeuts} show that in the case of a wedge, convergence fails after suitable chromatic localization, i.e., for every prime $p$ and every height $h$, the $v_h$-periodic Goodwillie calculus of a wedge diverges on spheres of dimension at least two, providing the first example of spaces which are not $\Phi_{K(h)}$-good in the sense of Behrens and Rezk~\cite{BehrensRezk}. The splitting of the calculus of a product together with the categorical properties of the Bousfield-Kuhn functor implies that the $v_h$-periodic Goodwillie tower of a product converges on suitably connected spaces, for details on $v_h$-periodic homotopy see Section~\ref{sec: convergence}. We show that among the simplicial complexes $K$ that we consider (those for which the so-called fat-wedge filtration is trivial) the product is the only example of a polyhedral product which has convergent $v_h$-periodic Goodwillie tower.

\begin{xxthm}\label{Thm: convergence}
If the fat wedge filtration on the simplicial complex $K \neq \Delta^{m-1}$ is trivial, then the $v_h$-periodic Goodwillie tower of
\[
((-), \underline{\ast})^K : \es{S}_\ast^{\times m} \longrightarrow \T,
\]
diverges on spheres of dimension at least two.
\end{xxthm}

Somewhat surprisingly, although our decompositions vary drastically, the existence of a decomposition is enough so that our proof of Theorem~\ref{Thm: convergence} follows the same lines as that of Brantner and Heuts. 

\subsection*{Future directions}
Our Theorem~\ref{Thm: multivariable decomp} and Theorem~\ref{Thm: convergence} raise many interesting questions about the relationship between Goodwillie calculus, chromatic homotopy theory and polyhedral products. We conclude this section by offering a sampling of these questions in the hopes of future discussions. 
 
In light of Theorem~\ref{Thm: convergence}, it seems that a generic simplicial complex `looks more like a wedge than a product' in the sense that it has divergent $v_h$-periodic Goodwillie tower. We wonder if this property is a consequence of having a decomposition of the Ganea fibre, and might fail for `wilder' simplicial complexes.

\begin{ques} 
Does there exist a simplicial complex $K\neq \Delta^{m-1}$, for which the $v_h$-periodic Goodwillie tower of
\[
((-), \underline{\ast})^K : \es{S}_\ast^{\times m} \longrightarrow \T,
\]
converges on spheres of dimension at least two?
\end{ques}

The polyhedral products we have considered have always been generated by pairs of the form $(X_i, \ast)$ or $(CX_i, X_i)$, i.e., pairs with a `single' variable. For many applications of polyhedral products, e.g., to toric geometry, one should consider more general pairs of the form $(X_i, A_i)$, where $A_i$ is a subspace of $X_i$, and hence the calculus of the functor 
\[
((-), (-))^K : \es{P}^{\times m} \longrightarrow \T, \ (\underline{X}, \underline{A}) \longmapsto (\underline{X}, \underline{A})^K,
\]
where $\es{P}$ is the category of pairs of pointed spaces. In~\cite[Theorem 2.3]{HaoSunTheriault}, Hao, Sun and Theriault provide a loop decomposition of general polyhedral products $(\underline{X}, \underline{A})$ in terms of the product and the polyhedral product $(C\underline{Y}, \underline{Y})$ where $Y_i$ is the fibre of the inclusion $A_i \to X_i$. 

\begin{ques}\label{Q: calc on pairs}
Is the Hao-Sun-Theriault loop decomposition natural and does a suitable notion of calculus exist for functors of the form $F: \es{P}^{\times m} \longrightarrow \T,$ so that the Goodwillie calculus of the functor $((-), (-))^K$ can be described in terms of Theorem~\ref{Thm: multivariable decomp} applied to $\underline{Y}$? 
\end{ques}

In many ways, one could interpret the content of our Theorem~\ref{Thm: multivariable decomp} as the statement: a loop decomposition of the functor $(-, \underline{\ast})^K$, induces a decomposition of the associated Goodwillie calculus in terms of the Goodwillie tower of the identity. 

\begin{ques}
Are there other loop decompositions of polyhedral products that give alternative decompositions of the Goodwillie tower? Theorem~\ref{Thm: multivariable decomp} is in absolute terms but one could try to construct a \emph{relative} decomposition in terms of more general subcomplexes. 
\end{ques}

\subsection*{Acknowledgements}
The authors were supported by the European Research Council (ERC) through the grant ``Chromatic homotopy theory of spaces'', grant no. 950048. We are grateful to Stephen Theriault for providing feedback on an earlier draft of this work, and Christian Carrick for patiently sharing an office with us during the production of this article.

\section{Goodwillie calculus}
Goodwillie calculus is a well-established tool for studying functors $F: \es{C} \to \es{D}$ between $\infty$-categories satisfying mild hypotheses, see e.g.,~\cite{GoodCalcI,GoodCalcII, GoodCalcIII} for the original exposition, and~\cite[\S6]{HA} for a modern account in full generality. In this section, we present a brief overview of the main concepts. The reader unfamiliar with $\infty$-categories is welcome to substitute the $\infty$-categories $\es{C}$ and $\es{D}$ for the category of pointed spaces, and replace any (co)limits by homotopy colimits. 

\subsection{Single-variable Goodwillie calculus}
Throughout this section we Let $\es{C}$ be an $\infty$-category with finite colimits and $\es{D}$ be a \emph{differential $\infty$-category}, i.e., an $\infty$-category with finite limits and filtered colimits which commute. The driving force behind single-variable Goodwillie calculus is the categorification of differential calculus on the $\infty$-category $\Fun(\es{C}, \es{D})$ of functors from $\es{C}$ to $\es{D}$.  

Denote by $\es{P}(n)$ the poset of subsets of $[n]$. An \emph{$n$-cube} in $\es{C}$ is a functor $X: \es{P}(n) \to \es{C}$. An $n$-cube $X$ is said to be \emph{strongly cocartesian} if it is left Kan extended from the poset $\es{P}(n)_{\leq 1}$ of subsets of $[n]$ of cardinality at most one along the canonical inclusion $\es{P}(n)_{\leq 1} \hookrightarrow \es{P}(n)$, i.e., if each face is a pushout. An $n$-cube is \emph{(co)cartesian} if it is a (co)limit diagram. A functor $F: \es{C} \to \es{D}$ is \emph{$n$-excisive} if it sends strongly cocartesian $(n+1)$-cubes to cartesian $(n+1)$-cubes. We denote by $\exc{n}(\es{C}, \es{D})$ the full sub-$\infty$-category spanned by the $n$-excisive functors. The key theorem in the construction of the calculus is the following result of Goodwillie~\cite[\S1]{GoodCalcIII}, a proof of which in full generality is provided by Lurie~\cite[Theorem 6.1.1.10]{HA}.

\begin{thm}\label{thm: existence excisive}
For each $n \in \bb{N}$, the inclusion of $\infty$-categories
\[
\exc{n}(\es{C}, \es{D}) \hookrightarrow \Fun(\es{C}, \es{D}),
\]
admits a left exact left adjoint
\[
\pushQED{\qed} 
P_n : \Fun(\es{C}, \es{D}) \longrightarrow \exc{n}(\es{C}, \es{D}).\qedhere
\popQED
\]   
\end{thm}

The upshot in practice is that one gets, for each functor $F : \es{C} \to \es{D}$, a functor $P_n F$, together with a natural transformation $F \to P_n F$, such that any natural transformation $F \to G$ for $G$ an $n$-excisive functor factors uniquely over $P_n F$. In other words, $P_n F$ is initial among $n$-excisive functors under $F$. 

\begin{rem}
Let $S$ be a finite set, and denote by $C_S(X)$ the $S$-pointed cone on $X\in \es{C}$ as in~\cite[
Construction 6.1.1.18]{HA}. Define $T_nF$ by the formula
\[
T_nF(X) = \lim_{S \in \es{P}_0(n+1)} F(C_S(X)).
\]
This definition is functorial in $F$ and $X$ and we may define $P_n$ to the fixed points of the action of $T_n$ on $F$, i.e.,
\[
P_nF(X) = \underset{q \in \bb{N}}{\colim}~T_n^qF(X).
\]
If $\es{C} = \T$, then the $S$-pointed cone is precisely what one would imagine, e.g., $C_\varnothing(X) = X$, $C_{[1]}X = CX$ and $C_{[2]}(X) =\Sigma X$.
\end{rem}

There is an inclusion of $\infty$-categories
\[
\exc{n}(\es{C}, \es{D}) \hookrightarrow \exc{(n+1)}(\es{C}, \es{D}),
\]
induced by the fact that every $n$-excisive functor is $(n+1)$-excisive, and it follows that for any functor $F: \es{C} \to \es{D}$, there is a tower of excisive approximations
\[
F \cdots \longrightarrow P_nF \longrightarrow P_{n-1}F \longrightarrow \cdots \longrightarrow P_1 F\longrightarrow P_0F,
\]
which we call the \emph{Goodwillie tower of $F$}.

\subsection{Multi-variable Goodwillie calculus}
Our attention now turns to the Goodwillie calculus of functors of the form $F: \es{C}_1 \times \cdots \times \es{C}_m \to \es{D}$. Here each $\es{C}_i$ should have finite colimits and $\es{D}$ should be a differential $\infty$-category. In our applications, we will work predominately with multi-variable functors on pointed spaces and encourage the reader less familiar with $\infty$-categories to keep this example in mind. Full details in full generality can be found in~\cite[\S6.1.3]{HA}, for the case of multi-variable functors of spaces, a short account is provided in~\cite{AroneKankaanrinta}.

Let $\Vec{n}=(n_1, \dots, n_m) \in \bb{N}^{\times m}$ be a multi-index. A multi-variable functor $F: \es{C}_1 \times \cdots \times \es{C}_m \to \es{D}$ is \emph{$\Vec{n}$-excisive} if for each $1\leq i \leq n$ it is $n_i$-excisive in the $i$-th variable, i.e., if each functor
\[
\es{C}_i \hookrightarrow \es{C}_i \times \prod_{j\neq i} \{X_j\} \hookrightarrow \prod_{j=1}^m \es{C}_j \xrightarrow{F} \es{D},
\]
is $n_i$-excisive. Denote by $\exc{\Vec{n}}(\prod_{i=1}^m \es{C}_i, \es{D})$ the full sub-$\infty$-category of $\Fun(\prod_{i=1}^m \es{C}_i, \es{D})$ spanned by the $\Vec{n}$-excisive functors. By definition, 
\[
\exc{\Vec{n}}(\textstyle\prod_{i=1}^m \es{C}_i, \es{D}) \simeq \exc{n_1}(\es{C}_1, \exc{\Vec{n'}}(\prod_{i=2}^m \es{C}_i, \es{D}),
\]
where $\Vec{n'} = (n_2, \dots, n_m)$, thus one may inductively construct a universal $\Vec{n}$-excisive approximation functor, see e.g.,~\cite[Remark 1.22]{GoodCalcIII} and~\cite[Proposition 6.1.3.6]{HA}.

\begin{thm}\label{thm: existence multi-variable excisive}
The inclusion of $\infty$-categories
\[
\exc{\Vec{n}}(\textstyle\prod_{i=1}^m \es{C}_i, \es{D})  \hookrightarrow \Fun(\prod_{i=1}^m \es{C}_i, \es{D}),
\]
has a left exact left adjoint
\[
\pushQED{\qed} 
P_{\Vec{n}} : \Fun(\textstyle\prod_{i=1}^m \es{C}_i, \es{D}) \longrightarrow \exc{\Vec{n}}(\prod_{i=1}^m \es{C}_i, \es{D}).\qedhere
\popQED
\] 
\end{thm}

\begin{rem}
We can give a construction of $P_{\Vec{n}}$ analogous to the construction of the universal $n$-excisive approximation in the single-variable situation: define 
\[
T_{\Vec{n}}F(X_1, \dots, X_m) = \lim_{(S_i) \in \prod_{i=1}^m \es{P}_0(n_i+1)} F(C_{S_1}(X_1), \dots, C_{S_m}(X_m)),
\]
and define 
\[
P_{\Vec{n}}F(X_1, \dots, X_m) = \underset{q \in \bb{N}}{\colim}~T_{\Vec{n}}^qF(X_1, \dots, X_m).
\]
\end{rem}

\begin{rem}
An alternative definition of the universal $\Vec{n}$-excisive approximation is given by Arone and Kankaanrinta~\cite{AroneKankaanrinta}. They define the universal $\Vec{n}$-excisive approximation of $F$ as
\[
T_{\Vec{n}}^{\Vec{\infty}}F(X_1, \dots, X_m) = \underset{(q_i) \in \bb{N}^{\times m}}{\colim}~T_{\Vec{n}}^{(q_1, \dots, q_m)}F(X_1, \dots, X_m).
\]
The ordering on multi-indices (see below) tells us that these definitions agree (up to homotopy) by a (right) cofinality argument, see e.g.,~~\cite[\href{https://kerodon.net/tag/03U2}{Corollary 03U2}]{kerodon}.
\end{rem}

Multi-indices come with a canonical ordering: $\Vec{k}=(k_1, \dots, k_m) \leq \Vec{n} = (n_1, \dots, n_m)$ if $k_i \leq n_i$ for each $1 \leq i \leq m$. Denoting by $\Vec{n}-1$ the multi-index $(n_1-1, \dots, n_m-1)$, we see that $\Vec{n}-1 \leq \Vec{n}$. This induces a map $P_{\Vec{n}}F \to P_{\Vec{n}-1}F$ which replaces the Goodwillie tower in this multi-variable setting. 

\begin{rem}
Given a multi-variable functor $F: \prod_{i=1}^m \es{C}_i \to \es{D}$, the single variable and multi-variable excisive approximations to $F$ commute. Since both are filtered colimits of finite limit constructions it suffices to show that for $k \in \bb{N}$ and $\Vec{n}$ a multi-indexed there is an equivalence $T_k T_{\Vec{n}}F \simeq T_{\Vec{n}}T_kF$. This can readily be checked using the constructions of $T_k$ and $T_{\Vec{n}}$ as homotopy limits of punctured cubes and that the $S$-pointed cone construction on $\prod_{i=1}^m \es{C}_i$ is defined component-wise.
\end{rem}

\section{Integral loop space decompositions of the polyhedral product}\label{sec: integral decomps}

Write $[m]=\{1, 2, \dots, m\}$ for the set of $m$ elements. Suppose given a sequence of pairs of spaces $(\underline{X}, \underline{A}):=\{(X_i,A_i)\}_{i \in [m]}$, where for each $i \in [m]$ the space $A_i$ is a subspace of $X_i$. Let $K$ be a simplicial complex on the vertex set $[m]$. For each simplex $\sigma \in K$, we write
\[
(\underline{X}, \underline{A})^\sigma := \prod_{i=1}^m W_i, \hspace{3ex} \text{where} \hspace{3ex}  W_i = 
\begin{cases} 
X_i & i \in \sigma \\ 
A_i & i \not \in \sigma. 
\end{cases}
\]
The \emph{polyhedral product} of $(\underline{X},\underline{A})$ with respect to $K$ is then defined to be
\[
(\underline{X},\underline{A})^K = \bigcup_{\sigma \in K} (\underline{X}, \underline{A})^\sigma \subset \prod_{i=1}^m X_i.
\]
For a given $K$, our arguments depend on having a decomposition of the polyhedral product functor 
\[
\Omega(-, \underline{\ast})^K: \es{S}_\ast^{\times m} \longrightarrow \T, \ \underline{X} \longmapsto \Omega(\underline{X}, \underline{\ast})^K,
\] 
that can play the role that the Hilton-Milnor theorem plays in the case of a wedge \cite{BrantnerHeuts}. The decomposition we will use is (in full generality) a combination of the work of Denham-Suciu \cite{DenhamSuciu} and Iriye-Kishimoto \cite{IriyeKishimoto2}, which we summarise here.

\subsection{A split fibre sequence}

First, we will need the following, which is essentially a corollary of a result of Denham and Suciu~\cite[Lemma 2.3.1]{DenhamSuciu} noting that $(\underline{X},\underline{X})^K = \prod_{i=1}^m X_i$. The formalism we give is due to Iriye and Kishimoto~\cite[Proposition 2.6]{IriyeKishimoto}.

\begin{lem}\label{lem: Denham Suciu sequence}
Let $K$ be a simplicial complex on $[m]$. There is a fibre sequence
\[
(C\Omega\underline{X},\Omega\underline{X})^K \longrightarrow (\underline{X},\underline{\ast})^K \longrightarrow \prod_{i=1}^m X_i,
\]
natural in $\underline{X}$. \qed
\end{lem}

It is standard that the inclusion $\bigvee_{i=1}^m X_i \to \prod_{i=1}^m X_i$ has a natural left homotopy inverse after looping. If $K$ contains each vertex $i \in [m]$, then this inclusion factors through $(\underline{X},\underline{\ast})^K$, so the map $(\underline{X},\underline{\ast})^K \to \prod_{i=1}^m X_i$ also has a natural left homotopy inverse after looping. Combining this left homotopy inverse with the map $\Omega (C\Omega\underline{X},\Omega\underline{X})^K  \to \Omega (\underline{X},\underline{\ast})^K$ using the loop multiplication on the target gives the following standard corollary.

\begin{cor}\label{cor: cone fibre seq splits} 
Let $K$ be a simplicial complex on $[m]$. There is a homotopy equivalence 
\[
\Omega (C\Omega\underline{X},\Omega\underline{X})^K \times \prod_{i=1}^m \Omega X_i \xrightarrow{\ \sim \ } \Omega (\underline{X},\underline{\ast})^K,
\]
natural in $\underline{X}$. \qed
\end{cor}

It follows that to identify the homotopy type of the right-hand side, it suffices to identify the homotopy type of polyhedral products of the form $(C\underline{Y},\underline{Y})^K$. This has been done for a fairly large class of simplicial complexes.

\subsection{The decomposition of Iriye-Kishimoto} 
The polyhedral product $(D^1, S^0)^K$ is called the \emph{real moment-angle complex on $K$}, and it plays a key role in the integral decompositions. The skeletal filtration of the full $(m-1)$-simplex provides a filtration of the product
\[
\ast = (\underline{X}, \underline{\ast})^{\sf{sk}_{-1}\Delta^{m-1}} \hookrightarrow (\underline{X}, \underline{\ast})^{\sf{sk}_0\Delta^{m-1}}  \hookrightarrow \cdots \hookrightarrow (\underline{X}, \underline{\ast})^{\sf{sk}_{m-1}\Delta^{m-1}} = \prod_{i=1}^m X_i,
\]
which we call the \emph{fat wedge filtration}, following Iriye and Kishimoto~\cite{IriyeKishimoto2}. For any subspace $Y \subset \prod_{i=1}^m X_i,$ we obtain a filtration
\[
\ast = Y \cap (\underline{X}, \underline{\ast})^{\sf{sk}_{-1}\Delta^{m-1}} \hookrightarrow Y \cap (\underline{X}, \underline{\ast})^{\sf{sk}_0\Delta^{m-1}}  \hookrightarrow \cdots \hookrightarrow Y \cap(\underline{X}, \underline{\ast})^{\sf{sk}_{m-1}\Delta^{m-1}} = Y \cap\prod_{i=1}^m X_i,
\]
called the \emph{fat wedge filtration on $Y$}. For $Y=(D^1, S^0)^K$ and $X_i = D^1$ we obtain the \emph{fat wedge filtration on the real moment angle complex}. Iriye and Kishimoto~\cite[Theorem 3.1]{IriyeKishimoto2} prove that the $k$-th stage of the fat wedge filtration on the real moment angle complex is obtained from the $(k-1)$-st stage by attaching cones, and we say that the fat wedge filtration is \emph{trivial} if these attaching maps are null. The fat wedge filtration on the real moment angle complex is trivial for a large class of simplicial complexes see e.g.,~\cite[\S7-10]{IriyeKishimoto2}, in particular, for \emph{shifted} simplicial complexes: whenever $\sigma$ is a simplex of $K$, $i$ is a vertex of $\sigma$, and $j>i$, then the simplex $(\sigma \setminus \{i\}) \cup \{j\}$ obtained by replacing $i$ with $j$ is also a simplex of $K$.

Bahri, Bendersky, Cohen, and Gitler provided~\cite[Theorem 2.10]{BBCGFunctor} an integral decomposition of polyhedral products on any simplicial complex $K$ after a single suspension and conjectured~\cite[Conjecture 2.29]{BBCGFunctor} that their decomposition desuspend when $K$ is a shifted complex. This conjecture was proven independently by Grbi\'{c}-Theriault \cite{GrbicTheriault} and Iriye-Kishimoto \cite{IriyeKishimoto}, generalizing classical results of Porter~\cite{PorterHOWP}. The result was then further generalized by Iriye and Kishimoto \cite{IriyeKishimoto2} to simplicial complexes $K$ for which the fat wedge filtration of the real moment-angle complex is trivial. Our application depends very strongly on knowing that the decomposition we have is natural in $\underline{X}$. This naturality is not included in the statement of Iriye and Kishimoto's main theorem, but the more detailed assertion below is included in the proof \cite[Proof of Theorem 1.2]{IriyeKishimoto2}.

\begin{thm} \label{thm: shifted complex homotopy type}
Let $K$ be a simplicial complex on $[m]$. Suppose that the fat wedge filtration of the real moment-angle complex on $K$ is trivial. There is a functor $D: \es{S}_\ast^{\times m} \to \T$, depending on $K$, and a zigzag of homotopy equivalences of spaces
\[
(C\underline{X},\underline{X})^K \xleftarrow{\ \sim\ } D(\underline{X}) \xrightarrow{ \ \sim \ } \bigvee_{\emptyset \neq I \subseteq [m]} |\Sigma K_I| \wedge \widehat{X}^I,
\]
where $\widehat{X}^I = \bigwedge_{i \in I} X_i$, which is natural in $\underline{X}$. \qed
\end{thm}

\begin{rem} The functor $D$ is complicated: an explicit description of it can be extracted from \cite{IriyeKishimoto2}, but this is unnecessary for our purposes.
\end{rem}

Combining Corollary \ref{cor: cone fibre seq splits} and Theorem \ref{thm: shifted complex homotopy type} now gives a decomposition of $\Omega (\underline{X},\underline{\ast})^K$ which is adequate for our purposes. In the case that $m=2$ and $K$ consists of the two disjoint points this is the classical description of the fibre of $X \vee Y \to X \times Y$, due to Ganea \cite{Ganea}.

We will now put the decomposition together for the polyhedral product $\Omega(C\Omega\underline{X}, \Omega\underline{X})^K$ which splits off $\Omega(\underline{X}, \underline{\ast})^K$. For the first step, we need to suspend the input spaces, but we do not yet need to loop on the outside.

\begin{lem} \label{lem: James splitting the decomp} If $K$ is a simplicial complex on $[m]$, and $I = (i_1, \dots, i_\ell) \subseteq [m]$, then there is a homotopy equivalence
$$ |\Sigma K_I| \wedge \bigvee_{(k_1, \dots,k_\ell) \in \mathbb{Z}^\ell_{\geq 1}} X_{i_1}^{\wedge k_1} \wedge \dots \wedge X_{i_\ell}^{\wedge k_\ell} \xrightarrow{\ \sim \ } |\Sigma K_I| \wedge \widehat{\Omega \Sigma X}^I,$$
which is natural in $\underline{X}$.    
\end{lem}

\begin{proof} Move the suspension coordinate to each factor in turn, and apply the James decomposition, see e.g.,~\cite[Proposition 4.18]{DevalapurkarHaine}, which gives a natural homotopy equivalence 
\[
\Sigma \bigvee_{k \geq 1} X^{\wedge k} \to \Sigma \Omega \Sigma X. \qedhere
\]
\end{proof}

We now have a wedge of suspensions and can use the Hilton-Milnor theorem to get a description of the loop space. We will introduce notation for the wedge summands in the previous theorem since these will be the generators of our Lie algebra.

\begin{definition} \label{def: alpha} Let $K$ be a simplicial complex on $[m]$, let $I = (i_1, \dots, i_\ell)$ be a subset of $[m]$, and let $\Vec{k} = (k_i)_{i \in I}$. Define $\alpha_{I,\Vec{k}} : \es{S}_\ast^{\times m} \to \T$ be the functor 
\[
\alpha_{I,\Vec{k}} : \underline{X} \longmapsto |K_I| \wedge \bigwedge_{i \in I} X_{i}^{ \wedge k_i}.
\]
\end{definition}

The functors $\alpha_{I, \Vec{k}}$ are indexed on the set whose elements are pairs $(I,\Vec{k})$, where $I$ is a nonempty subset of $[m]$, and $\Vec{k} \in \mathbb{Z}^I_{\geq 1}$. Note that the functor $\alpha_{I,\Vec{k}}$ depends implicitly on the simplicial complex $K$, but we omit this dependence from the notation. Lemma \ref{lem: James splitting the decomp} can now be rewritten as follows.

\begin{cor} \label{cor: decomp as alpha} If $K$ is a simplicial complex on $[m]$, and $I \subseteq [m]$, then there is a homotopy equivalence
\[
\Sigma \bigvee_{\Vec{k} \in \mathbb{Z}^I_{\geq 1}} \alpha_{I,\Vec{k}}(\underline{X}) \xrightarrow{ \ \sim \ } |\Sigma K_I| \wedge \widehat{\Omega \Sigma X}^I,
\]
which is natural in $\underline{X}$.  \qed
\end{cor}

Looping on the outside now allows us to apply the Hilton-Milnor theorem to the decomposition obtained by combining Corollary \ref{cor: decomp as alpha} and Theorem \ref{thm: shifted complex homotopy type}. First, we note that the Hilton-Milnor theorem generalizes without additional difficulty to countable wedges, as follows. 

\subsection{The Hilton-Milnor theorem}
Let $\mathscr{I}$ be a countable set, and take the free lie algebra $L(\mathscr{I})$ to be the free Lie algebra on symbols $\{y_i\}_{i \in \mathscr{I}}$. If one gives $\mathscr{I}$ a total ordering, $L(\mathscr{I})$ is the colimit of the free Lie algebras $L(\mathscr{I}_{\leq i})$ over $i \in \mathscr{I}$ (where $\mathscr{I}_{\leq i}$ is the set of $j \leq i$ in $\mathscr{I}$).  Compatible Hall bases for the $L(\mathscr{I}_{\leq i})$ can be chosen, and then a Hall basis $\bb{L}(\mathscr{I})$ for $L(\mathscr{I})$ is obtained as the union of the Hall bases for the $L(\mathscr{I}_{\leq i})$. A word $\omega \in \bb{L}(\mathscr{I})$ has only finite length, so can contain only finitely many of the letters $y_i$. Write $a(\omega)$ for the number of letters appearing in $\omega$ ($a$ for `arguments'), and write these letters $y_{i_1} , \dots , y_{i_{a(\omega)}}$, where $i_1 < \dots < i_{a(\omega)}$ in the ordering on $\mathscr{I}$. 

\begin{definition} \label{def: omega prime} For each $\omega$, we obtain a functor $\omega : \es{S}_\ast^{\mathscr{I}} \to \T$ which acts on a countable collection of spaces $\underline{Y}=\{Y_i\}_{i \in \mathscr{I}}$ by smashing the arguments in the order given by $\omega$ (neglecting brackets). Note that this functor depends only on the finitely many $Y_i$ for which $y_i$ actually appears in the word $\omega$, so it factors over $\T^{\times a(\omega)}$. We will freely identify $\omega$ with this factorization; this amounts to being allowed to write $$\omega(Y_{i_1}, \dots , Y_{i_{a(\omega)}}) = \omega(\underline{Y}).$$
\end{definition}

With this notation we have the following lemma.

\begin{lem} \label{lem: countable HM} Let $\underline{Y}=\{Y_i\}_{i \in \mathscr{I}}$ be a countable collection of spaces. If $Y_i$ is connected for each $i$, then there exists a zigzag of homotopy equivalences
$$\Omega \sideset{}{'}\prod_{\omega \in \bb{L}(\mathscr{I})} \Sigma \omega(\underline{Y}) \simeq \Omega \Sigma \bigvee_{i \in \mathscr{I}} Y_i,$$
which is natural in $\underline{Y}$, where $\bb{L}(\mathscr{I})$ is a basis for the free Lie algebra on the collection of symbols $ \{y_i\}_{i \in \mathscr{I}}$. 
\end{lem}

\begin{proof} 
The statement for finite collections of spaces is just (a natural version of) the Hilton-Milnor theorem, see e.g.,~\cite[Proposition 4.21]{DevalapurkarHaine} or~\cite[Theorem 5.9]{Lavenir}, so what must be established is that this implies the statement for countable collections. This will essentially follow from the fact that filtered colimits commute with finite limits.

Ordering the indexing set $\mathscr{I}$ as before, the ordinary Hilton-Milnor map gives a natural homotopy equivalence
$$
\Omega \sideset{}{'}\prod_{\omega \in \bb{L}(\mathscr{I}_{\leq i})} \Sigma \omega(\underline{Y}) \xrightarrow{\ \sim \ } \Omega \Sigma \bigvee_{j \leq i} Y_i,
$$ 
for each $i$. Taking the colimit over $i$, since $\Omega$ is a finite limit and hence commutes with filtered colimits, the left-hand side becomes 
\begin{align*}
     \underset{i \in \mathscr{I}}{\colim}~\Omega \sideset{}{'}\prod_{\omega \in \bb{L}(\mathscr{I}_{\leq i})} \Sigma \omega(\underline{Y}) & \xrightarrow{\ \sim \ } \Omega~  \underset{i \in \mathscr{I}}{\colim}~ \sideset{}{'}\prod_{\omega \in \bb{L}(\mathscr{I}_{\leq i})} \Sigma \omega(\underline{Y}) \\
    & \xrightarrow{\ = \ } \Omega~  \underset{i \in \mathscr{I}}{\colim}~\underset{{\substack{S \subset  \bb{L}(\mathscr{I}_{\leq i}) \\ |S| < \infty}}}{\colim}~\prod_{\omega \in S} \Sigma \omega(\underline{Y}) \\
    & \xleftarrow{ \ \sim \ }  \Omega~\underset{{\substack{S \subset \bigcup_i \bb{L}(\mathscr{I}_{\leq i}) \\ |S| < \infty}}}{\colim}~\prod_{\omega \in S} \Sigma \omega(\underline{Y}) \\
    & \xrightarrow{ \ = \ } \Omega~ \underset{{\substack{S \subset \bb{L}(\mathscr{I}) \\ |S| < \infty}}}{\colim}~\prod_{\omega \in S} \Sigma \omega(\underline{Y}),
\end{align*}
where the third equivalence is by cofinality, and the right-hand side becomes
\[
    \underset{i \in \mathscr{I}}{\colim}~ \Omega \Sigma \bigvee_{j \leq i} Y_i \xrightarrow{ \ \sim \ }  \Omega \Sigma ~\underset{i \in \mathscr{I}}{\colim}~\bigvee_{j \leq i} Y_i = \Omega \Sigma \bigvee_{i \in \mathscr{I}} Y_i,
\]
as required.
\end{proof}

Now suppose given a countable collection of functors $F = \{F_i : \es{S}_\ast^{\times m} \to \T\}_{i \in \mathscr{I}}$. Setting $Y_i = F_i(\underline{X})$ in Lemma \ref{lem: countable HM}, we obtain the following.

\begin{cor} \label{cor: countable composed HM} Let $F= \{F_i\}_{i \in \mathscr{I}}$ be a countable collection of functors $\es{S}^{\times m}_\ast \to \T$. If $F_i(\underline{X})$ is connected for each $i$, then there exists a zigzag of homotopy equivalences
$$\Omega \sideset{}{'}\prod_{\omega \in \bb{L}(\mathscr{I})} \Sigma \omega F(\underline{X}) \simeq \Omega \Sigma \bigvee_{i \in \mathscr{I}} F_i(\underline{X}),$$
which is natural in $\underline{X}$, where $\bb{L}(\mathscr{I})$ is a basis for the free Lie algebra on symbols $\{y_i \}_{i \in \mathscr{I}}$. \qed
\end{cor}

Let $\alpha =\{ \alpha_{I, \Vec{k}}\}$ be the set of functors from Definition~\ref{def: alpha} indexed on the set whose elements are pairs $(I,\Vec{k})$, where $I$ is a nonempty subset of $[m]$, and $\Vec{k} \in \mathbb{Z}^I_{\geq 1}$. Recall that we call this indexing set $\mathscr{I}_m$, and write $\alpha$ for the collection $\{\alpha_{(I, \vec{k})}\}_{(I,\vec{k}) \in \mathscr{I}_m}$.

\begin{prop} \label{prop: integral decomposition of the fibre} If $K$ is a simplicial complex on $[m]$, and each $X_i$ is connected, then (with the preceding notation) there is a zigzag of homotopy equivalences
$$ \Omega \sideset{}{'}\prod_{\omega \in \bb{L}(\mathscr{I}_m)} \Sigma \omega \alpha(\underline{X}) \simeq \Omega \bigvee_{\emptyset \neq I \subseteq [m]} |\Sigma K_I| \wedge \widehat{\Omega \Sigma X}^I$$
which is natural in $\underline{X}$.
\end{prop} 
\begin{proof} By Corollary \ref{cor: decomp as alpha} we have a natural homotopy equivalence
$$\Sigma \bigvee_{\emptyset \neq I \subseteq [m]} \bigvee_{\Vec{k} \in \mathbb{Z}^I_{\geq 1}} \alpha_{I,\Vec{k}}(\underline{X}) \xrightarrow{\sim} \bigvee_{\emptyset \neq I \subseteq [m]} |\Sigma K_I| \wedge \widehat{\Omega \Sigma X}^I.$$

Since each $X_i$ is connected, the spaces $\alpha_{I,\Vec{k}}(\underline{X})$ are connected, so the result follows Corollary \ref{cor: countable composed HM}.
\end{proof}

Note in particular that $\omega \alpha$ carries $\underline{X}$ the smash product of some number of copies of each $X_i$, together with some number of copies of the fixed spaces $|K_I|$.

\section{Goodwillie approximations of polyhedral products}

\subsection{Multi-variable approximations} In this section, we study the multi-variable calculus of polyhedral products of the form $(\underline{X}, \underline{\ast})^K$. In order to use the decompositions that we have, we must introduce a single loop, and study the multi-variable calculus of the functor 
\[
\Omega(-, \underline{\ast})^K: \es{S}_\ast^{\times m} \longrightarrow \T, \ \underline{X} \longmapsto \Omega(\underline{X}, \underline{\ast})^K.
\]
Since the excisive approximations are left exact, the insertion of an additional loop coordinate doesn't introduce any technical problems. We begin with the observation that we can split off the product factor. The calculus of products is easily understood using the left exactness of the excisive approximations.

\begin{lem}\label{lem: decomposition of multi-variable excisive}
Let $K$ be a simplicial complex on $[m]$ and let $\Vec{n}$ be a multi-index. There is an equivalence
\[
P_{\Vec{n}}(\Omega(-, \underline{\ast})^K)(\underline{X}) \simeq \Omega P_{\Vec{n}}((C\Omega(-), \Omega(-))^K)(\underline{X}) \times \Omega P_{\Vec{n}}(\textstyle\prod_{i=1}^m(-))(\underline{X}).
\]
natural in $\underline{X}$.
\end{lem}
\begin{proof}
Apply the fact that excisive approximations are left exact hence preserve split fibre sequences to the natural splitting of Corollary~\ref{cor: cone fibre seq splits}. 
\end{proof}

It follows that to understand the multi-variable calculus of polyhedral products of the form $(\underline{X}, \underline{\ast})^K$ we need a understanding of the multi-variable calculus of the product $\prod_{i=1}^m X_i$, and of the polyhedral product of the form $(C\Omega\underline{X}, \Omega\underline{X})^K$. We now treat each piece of this loop space decomposition on the level of calculus separately. 

\begin{lem}\label{lem: multi-variable approx to product}
Let $\Vec{n} = (n_1, \dots, n_m)$ be a multi-index. There is an equivalence, 
\[
P_{\Vec{n}}(\prod_{i=1}^m )(\underline{X}) \simeq \prod_{i=1}^m P_{n_i}(\id)(X_i),
\]
natural in $\underline{X}$.
\end{lem}
\begin{proof}
The universal $\Vec{n}$-excisive approximation preserves products, so there is a natural equivalence
\[
P_{\Vec{n}}(\prod_{i=1}^m) \simeq \prod_{i=1}^m P_{\Vec{n}}(\id \circ~\sf{pr}_i),
\]
and the result follows since the projection map $\sf{pr}_i$ preserves colimits, and $P_{\Vec{n}}$ is given by preforming the universal $n_i$-excisive approximation in the $i$-th variable.
\end{proof}

For the factor $(C\Omega\underline{X}, \Omega\underline{X})^K$, we additionally need to suspend on the inside. Recall that $\alpha = \{\alpha_{I, \Vec{k}}\}$ is the set of functors from Definition~\ref{def: alpha} indexed on the set $\mathscr{I}_m$ of pairs $(I,\Vec{k})$, where $I$ is a non-empty subset of $[m]$ and $\vec{k} \in \mathbb{Z}^I_{\geq 1}$. For any $F: \es{S}_\ast^{\times m} \to \T$, there is an equivalence $P_{\Vec{n}}(\Omega F \Sigma) \simeq \Omega P_{\Vec{n}}(F) \Sigma$, so the left-hand term in the equivalence of the theorem below is also a description of the multi-variable calculus of the functor that sends $\underline{X}$ to $\Omega(C\Omega\Sigma\underline{X},\Omega\Sigma\underline{X})^K$, hence in particular of one factor of $\Omega(\underline{\Sigma X},\underline{\ast})^K$ (c.f. Corollary \ref{cor: cone fibre seq splits}).

\begin{thm} \label{thm: multivariable calc of fibre} Let $K$ be a simplicial complex on $[m]$ for which the fat wedge filtration of the real moment-angle complex on $K$ is trivial, and let $\underline{X} \in \es{S}^{\times m}_\ast$ with each $X_i$ connected. Let $\Vec{n} = (n_1, \dots n_m)$ be a multi-index. There is a natural equivalence
\[
\Omega P_{\Vec{n}} (C\Omega(-), \Omega(-))^K (\Sigma(\underline{X})) \simeq \Omega \sideset{}{'}\prod_{\omega \in \bb{L}(\mathscr{I}_m)} P_\kappa(\id)(\Sigma \omega \alpha(\underline{X})),
\]
where $\kappa = \min_i(\lfloor \frac{n_i}{a_i} \rfloor)$ for $a_i$ the smash power of $X_i$ appearing in the word $\omega\alpha(\underline{X})$.
\end{thm}

\begin{proof} We will apply calculus to the integral decomposition obtained by combining Theorem \ref{thm: shifted complex homotopy type} and Proposition \ref{prop: integral decomposition of the fibre}. We write:
    \begin{align*}
        \Omega P_{\Vec{n}} (C\Omega(-), \Omega(-))^K (\Sigma(\underline{X})) & \simeq P_{\Vec{n}}[ \Omega (C\Omega(-), \Omega(-))^K \Sigma ](\underline{X})\\
        & \simeq P_{\Vec{n}} [\Omega \sideset{}{'}\prod_{\omega \in \bb{L}(\mathscr{I}_m)} \Sigma \omega \alpha](\underline{X}) \\
        & \simeq \Omega \sideset{}{'}\prod_{\omega \in \bb{L}(\mathscr{I}_m)} P_{\Vec{n}} [  \Sigma \omega \alpha](\underline{X}),
    \end{align*} where for the last equivalence we use that $P_{\Vec{n}}$ commutes with finite limits and filtered colimits as a left exact left adjoint. Recalling that $\omega \alpha = \omega(\alpha_{I_1,\Vec{k}_1}, \dots ,\alpha_{I_{a(\omega)},\Vec{k}_{a(\omega)}})$ is just a smash power, the result now follows from \cite[Lemma 2.4]{BrantnerHeuts} (originally due to Arone-Kankaanrinta \cite[Proof of Lemma 1.4]{AroneKankaanrinta}).
\end{proof}

In particular, note that $\kappa$ is implicitly a function of $\omega$ and $\alpha_{I_1, \vec{k}_1}, \dots, \alpha_{I_{a(\omega)}, \vec{k}_{a(\omega)}}$ (c.f. Definition \ref{def: omega prime}), so it is takes different values in different factors of the weak product. We leave this dependency implicit because the notation is already overburdened.

A statement of Theorem \ref{thm: multivariable calc of fibre} where the effect of the complex $K$ is more evident is as follows.

\begin{definition} \label{def: IK}
    Let $\mathscr{I}_K$ be the subset of $\mathscr{I}_m$ indexed on those $(I,\Vec{k})$ for which $I \not \in K$; precisely, $\mathscr{I}_K$ is the set of pairs $(I , \vec{k})$ where $I$ is a nonempty subset of $[m]$ not contained in $K$ and $\vec{k} \in \mathbb{Z}^I_{\geq 1}$.
\end{definition}

\begin{cor} \label{cor: multivariable calc of fibre, restricted indexing} Let $K$ be a simplicial complex on $[m]$ for which the fat wedge filtration of the real moment-angle complex on $K$ is trivial, and let $\underline{X} \in \es{S}^{\times m}_\ast$ with each $X_i$ connected. Let $\Vec{n} = (n_1, \dots n_m)$ be a multi-index. There is an equivalence
$$\Omega P_{\Vec{n}} (C\Omega(-), \Omega(-))^K (\Sigma(\underline{X})) \simeq \Omega \sideset{}{'}\prod_{\omega \in \bb{L}(\mathscr{I}_K)} P_\kappa(\id)(\Sigma \omega \alpha)(\underline{X})),$$
where $\kappa = \min_i(\lfloor \frac{n_i}{a_i} \rfloor)$ for $a_i$ the smash power of $X_i$ appearing in the word $\omega \alpha (\underline{X})$. 
\end{cor}

\begin{proof} If $I \in K$, then $K_I \simeq *$, so for each $\vec{k}$, the functor $\alpha_{I,\Vec{k}}$ is canonically null. It follows that  if some $I_j \in K$ then the functor $\omega \alpha = \omega(\alpha_{I_1,\Vec{k}_1}, \dots ,\alpha_{I_{a(\omega)},\Vec{k}_{a(\omega)}})$ is canonically null, and these don't contribute to the weak product in Theorem \ref{thm: multivariable calc of fibre}. The result follows.
\end{proof}

The functor $P_0(\id)$ is the constant functor $X \mapsto *$, so we can now make the following observation.

\begin{cor}\label{cor: weak product is product after Pn}
    In Theorem \ref{thm: multivariable calc of fibre} and Corollary \ref{cor: multivariable calc of fibre, restricted indexing}, the index $\kappa$ can be nonzero for at most finitely many terms of the weak product, so all but finitely many terms are contractible, and the weak product may be replaced with the cartesian product.
\end{cor}

\begin{proof} Fix $\vec{n}$. We will first argue that there are only finitely many $\alpha_{I,\vec{k}}$ which can appear as arguments in words $\omega$ for which $\kappa > 0$. Write $b(I,\vec{k})_i$ for the smash power of $X_i$ appearing in $\alpha_{I,\vec{k}}$ (this is either zero or an entry of $\vec{k}$). There are only finitely many subsets $I$ of $[m]$, and for each $I$ there are only finitely many vectors $\vec{k}$ such that $\lfloor \frac{n_i}{b(I,\vec{k})_i} \rfloor$ is positive for every $i$. Call a pair $(I,\vec{k})$ \emph{potentially nonzero} if $\lfloor \frac{n_i}{b(I,\vec{k})_i} \rfloor$ is positive for every $i$; by the preceding considerations there are only finitely many potentially nonzero pairs. The smash power $a_i$ of $X_i$ appearing in $\omega \alpha = \omega(\alpha_{I_1,\Vec{k}_1}, \dots ,\alpha_{I_{a(\omega)},\Vec{k}_{a(\omega)}})$ is  $\sum_{j=1}^{a(\omega)}c_\omega(I_j,\vec{k}_j) b(I_j,\vec{k}_j)_i$, where $c_\omega(I_j,\vec{k}_j)$ is the number of instances of $\alpha_{I_j,\vec{k}_j}$ in the lie word $\omega$. It follows that 
\begin{enumerate}
    \item $\kappa$ can only be nonzero for words $\omega(\alpha_{I_1,\Vec{k}_1}, \dots ,\alpha_{I_{a(\omega)},\Vec{k}_{a(\omega)}})$ involving only arguments $\alpha_{I,\Vec{k}}$ for which $(I, \vec{k})$ is potentially nonzero, and
    \item for words involving only arguments $\alpha_{I,\Vec{k}}$ indexed by potentially nonzero pairs, $\kappa$ can only be nonzero if $c_\omega(I,\vec{k})$ is small enough (precisely, less than some $d(I,\vec{k})$ depending only on $(I,\vec{k})$).
\end{enumerate}

In other words, we have established that the factors for which $\kappa$ can be nonzero are indexed by a subset of the Hall basis of the free lie algebra on the finitely many possibly nonzero pairs, precisely, the subset of Hall words where each letter appears at most some finitely many times, depending on that basis element. The result then follows from the fact that a free lie algebra on finitely many letters contains only finitely many words of given word length.
\end{proof}

Combining Theorem~\ref{thm: multivariable calc of fibre}, Lemma~\ref{lem: multi-variable approx to product} and Lemma~\ref{lem: decomposition of multi-variable excisive} provides a proof of Theorem~\ref{Thm: multivariable decomp}.

\subsection{Single-variable approximations} We now briefly summarise the corresponding single-variable situation by precomposing with the diagonal functor $\Delta: \T \to \es{S}_\ast^{\times m}$. We state and prove the version of the theorem with the restricted indexing of Corollary \ref{cor: multivariable calc of fibre, restricted indexing} (c.f. Definition \ref{def: IK}) - the unrestricted version is equally true.

\begin{thm} \label{thm: single-variable calc of fibre, restricted indexing} Let $K$ be a simplicial complex on $[m]$ for which the fat wedge filtration of the real moment-angle complex on $K$ is trivial, and let $X \in \T$ be connected. Let $n \in \mathbb{Z}_{\geq 1}$. There is an equivalence
$$\Omega P_{n} (C\Omega(-), \Omega(-))^K (\Delta(\Sigma X)) \simeq \Omega \sideset{}{'}\prod_{\omega \in \bb{L}(\mathscr{I}_{K})} P_\kappa(\id)(\Sigma \omega \alpha \Delta(X)),$$
where $\kappa = \lfloor \frac{n}{a} \rfloor$ for $a = \Sigma_i a_i$, where $a_i$ is the smash power of the variable $X_i$ appearing in the word $\omega \alpha (\underline{X})$. 
\end{thm}

\begin{proof} We proceed as in the proof of Theorem \ref{thm: multivariable calc of fibre}, obtaining:
\begin{align*}
        \Omega P_{n} (C\Omega(-), \Omega(-))^K (\Delta(\Sigma X)) & \simeq \Omega \sideset{}{'}\prod_{\omega \in \bb{L}(\mathscr{I}_K)} P_{n} [  \Sigma \omega \alpha ](\Delta(X)).
    \end{align*}
Once again, we recall that $\omega \alpha \circ \Delta$ just forms a smash power of $X$, together with some number of copies of the complexes $K_{I_j}$. The result then follows by (the one-variable special case of) \cite[Lemma 2.4]{BrantnerHeuts}.
\end{proof}

The proofs of Lemma~\ref{lem: multi-variable approx to product} and Lemma~\ref{lem: decomposition of multi-variable excisive} extend as one would suspect to the single-variable calculus, and in combination with Theorem~\ref{thm: single-variable calc of fibre, restricted indexing}, proves a single-variable variant of Theorem~\ref{Thm: multivariable decomp} after precomposition with the diagonal.

As in the multivariable case, we have the following corollary for polyhedral products on $(\underline{\Sigma X}, \underline{*})$ (where all spaces $X_i$ must be equal):

\begin{cor} Let $K$ be a simplicial complex on $[m]$ for which the fat wedge filtration of the real moment-angle complex on $K$ is trivial, and let $X \in \T$ be connected. There is an equivalence
\[
\Omega P_{n}((\Sigma(-), \underline{\ast})^K \circ \Delta )(\underline{X}) \simeq \Omega \sideset{}{'}\prod_{\omega \in \bb{L}(\mathscr{I}_K)} P_\kappa(\id)(\Sigma \omega \alpha \Delta (X))\times \Omega \prod_{i=1}^m P_{n}(\id)(\Sigma X),
\]
where $\kappa = \lfloor \frac{n}{a} \rfloor$ for $a = \Sigma_i a_i$, where $a_i$ is the smash power of the variable $X_i$ appearing in the word $\omega \alpha (\underline{X})$. \qed \end{cor}

\subsection{Comparison with the identity functor evaluated on polyhedral products}
If one is just interested in polyhedral products of the form $(C\underline{X}, \underline{X})^K$, then a simpler decomposition is available directly from the Iriye-Kishimoto wedge decomposition of Theorem~\ref{thm: shifted complex homotopy type}. 

\begin{thm}\label{thm: calc decomposition cone on X}
Let $K$ be a simplicial complex on $[m]$ for which the fat wedge filtration of the real moment-angle complex on $K$ is trivial. For each $X_i$ connected, there is an equivalence 
\[
\Omega P_{\Vec{n}}((C(-), (-))^K)(\underline{X}) \simeq \Omega \sideset{}{'}\prod_{\omega \in \bb{L}(\es{P}_0(m))} P_\kappa(\id)(\Sigma\omega\beta(\underline{X})),
\]
natural in $\underline{X}$, where $\beta = \{\beta_I\}_{I \in \es{P}_0(m)}$ with $\beta_I(\underline{X}) = |K_I| \wedge \widehat{X}^{I}$, and $\kappa= \min_i(\lfloor \frac{n_i}{a_i} \rfloor)$ for $a_i$ the smash power of $X_i$ appearing in the word $\omega\beta(\underline{X})$.
\end{thm}
\begin{proof}
Proceed as in the proof of Theorem \ref{thm: multivariable calc of fibre} without invoking the James splittings. 
\end{proof}

We now apply the decomposition of the identity functor on a wedge of Branter and Heuts~\cite[Theorem 1.2]{BrantnerHeuts} to the wedge decomposition of Iriye and Kishimoto.

\begin{prop}\label{prop: BH for cone on X}
Let $K$ be a simplicial complex on $[m]$ for which the fat wedge filtration of the real moment-angle complex on $K$ is trivial. For each $X_i$ connected, there is an equivalence
\[
\Omega P_n(\id)((C\underline{X}, \underline{X})^K) \simeq \Omega \sideset{}{'}\prod_{\omega \in \bb{L}(\es{P}_0(m))}  P_{\lfloor \frac{n}{|\omega|} \rfloor}(\id)(\Sigma\omega \beta (\underline{X})),
\]
\end{prop}
\begin{proof}
By Theorem~\ref{thm: shifted complex homotopy type}, there is a natural equivalence 
\[
(C\underline{X}, \underline{X})^K \simeq \bigvee_{I \in \es{P}_0(m)} \Sigma|K_I| \wedge \widehat{X}^I,
\]
and hence a natural equivalence
\[
P_n(\id)((C\underline{X}, \underline{X})^K ) \simeq P_n(\id)(\bigvee_{I \in \es{P}_0(m)} |\Sigma K_I| \wedge \widehat{X}^I).
\]
An application of~\cite[Theorem 1.2]{BrantnerHeuts}, applied to the functor $\bigvee_{I \in\es{P}_0(m)} : \es{S}^{\es{P}_0(m)}_\ast \to \T$ yields the result.
\end{proof}

\begin{rem}\label{rem: compare with id}
The content of Theorem~\ref{thm: calc decomposition cone on X} and Proposition~\ref{prop: BH for cone on X} should be interpreted as a comment on the difference between the calculus of polyhedral products as a functor and the calculus of the identity applied to polyhedral products: the fact that polyhedral products do not, in general, preserve colimits results in a discrepancy between the two `calculi' approaches to polyhedral products. More precisely, the right hand side of the two theorems differs only in the Goodwillie degrees, and these degrees are smaller in Theorem~\ref{thm: calc decomposition cone on X}, so the tower there is a finer stratification: it converges slower. The mechanism is that the Goodwillie degrees in Theorem~\ref{thm: calc decomposition cone on X} experience a lag coming from $\beta$, wheras those in Proposition~\ref{prop: BH for cone on X} do not.
\end{rem}

\section{Convergence}\label{sec: convergence}

The Goodwillie tower of a functor $F: \es{C} \to \es{D}$ is said to \emph{converge at $X \in \es{C}$} if the canonical comparison map
\[
F(X) \longrightarrow \lim_{n \in \bb{N}^\op}~P_{n}F(X),
\]
is an equivalence. For multi-variable functors $F: \prod_{i=1}^m \es{C}_i \to \es{D}$, we say that the \emph{Goodwillie calculus converges at $\underline{X} \in \prod_{i=1}^m \es{C}_i$} if the comparison map
\[
F(\underline{X}) \longrightarrow \lim_{\Vec{n} \in  (\bb{N}^\op)^{\times m}}~P_{\Vec{n}}F(\underline{X}),
\]
is an equivalence. In this section, we relate multi-variable and single-variable convergence and prove that under our mild hypotheses on $K$, the polyhedral product functor $(-, \underline{\ast})^K$ converges integrally but not after suitable chromatic localization. 

\subsection{Multi-variable vs. single-variable convergence} We now show that multi-variable convergence is equivalent to single-variable convergence.

\begin{lem}\label{lem: convergence on diagonal}
Let $F: \prod_{i=1}^m \es{C}_i \to \es{D}$, For $n \in \bb{N}$ denote by $\Vec{v}_n$ the multi-index $(n, \dots, n)$. The multi-variable Goodwillie calculus of $F$ converges at $\underline{X}$ if and only if the canonical map
\[
    F(\underline{X}) \longrightarrow \lim_{n \in \bb{N}^\op}~P_{\Vec{v}_n}F(\underline{X}),
\]
is an equivalence.
\end{lem}
\begin{proof}
The claim follows from the fact that the diagonal functor
\[
\Vec{v}_{(-)}: \bb{N} \longrightarrow \bb{N}^{\times m}, \ n \longmapsto (n, \dots, n),
\]
is right cofinal, in the sense of~\cite[\href{https://kerodon.net/tag/02N1}{Definition 02N1}]{kerodon}, then the opposite functor will be left cofinal as required. By~\cite[\href{https://kerodon.net/tag/03U2}{Corollary 03U2}]{kerodon} the claim follows from the observation that for any $(n_1, \dots, n_m) \in \bb{N}^{\times m}$ there exists $n \in \bb{N}$ such that $(n_1, \dots, n_m) \leq \Vec{v}_n$, i.e., choose $n= n_1+\cdots+n_m$.
\end{proof}

\begin{rem}
We use the (co)finality terminology from~\cite{kerodon}, but the reader should proceed with caution as the terminology varies vastly between sources. 
\end{rem}

\begin{lem}\label{lem: vec n to k}
For any multi-index $\Vec{n} = (n_1, \dots, n_m)$ define $k=k(\Vec{n}) = \min\{n_i \mid 1 \leq i \leq m\}$. If $F: \prod_{i=1}^m \es{C}_i \to \es{D}$ is a multi-variable functor, then there is a natural transformation of multi-variable functors
\[
P_{\Vec{n}}F \longrightarrow P_{k}F,
\]
compatible with the universal natural transformations from $F$.
\end{lem}
\begin{proof}
It suffices to provide a natural transformation $T_{\Vec{n}}F \to T_{k}F$ which is compatible with the canonical natural transformation from $F$. By definition
\[
T_{\Vec{n}}F(X_1, \dots, X_m) = \lim_{(S_i) \in \prod_{i=1}^m \es{P}_0(n_i+1)} F(C_{S_1}(X_1), \dots, C_{S_m}(X_m)).
\]
For each $1 \leq i \leq m$ there is a canonical inclusion $\es{P}_0(k+1) \hookrightarrow \es{P}_0(n_i+1)$, which induces a natural map 
\[
T_{\Vec{n}}F(X_1, \dots, X_m) \longrightarrow \lim_{(S_i) \in \prod_{i=1}^m \es{P}_0(k+1)} F(C_{S_1}(X_1), \dots, C_{S_m}(X_m)).
\]
The diagonal map $\es{P}_0(k+1) \to \prod_{i=1}^m \es{P}_0(k+1)$ induces a further natural map
\[
\lim_{(S_i) \in \prod_{i=1}^m \es{P}_0(k+1)} F(C_{S_1}(X_1), \dots, C_{S_m}(X_m)) \longrightarrow \lim_{S \in  \es{P}_0(k+1)} F(C_{S}(X_1), \dots, C_{S}(X_m)),
\]
and hence to $T_kF(X_1, \dots, X_m)$ since $C_S(X_1, \dots X_m) = (C_S(X_1), \dots, C_S(X_m))$. The composite of these maps provides the required natural transformation
\[
T_{\Vec{n}}F \longrightarrow T_kF.
\]
Iterating provides a map
\[
T^q_{\Vec{n}}F \longrightarrow T^q_kF,
\]
compatible with $q$, and hence a map
\[
P_{\Vec{n}}F = \colim_qT^q_{\Vec{n}}F  \longrightarrow \colim_q T^q_kF = P_{k}F. \qedhere
\]
\end{proof}

\begin{lem}
Let $F: \prod_{i=1}^m \es{C}_i \to \es{D}$, and for each $n \in \bb{N}$, denote by $\Vec{v}_n$ the multi-index $(n, \dots, n)$ and by $\Vec{u}_n$ the multi-index $(\lfloor \frac{n}{m} \rfloor, \dots, \lfloor \frac{n}{m} \rfloor)$. There are maps of towers 
\[
\{P_{mn}F\}_{n \in \bb{N}^\op} \longrightarrow \{P_{\Vec{v}_n}F\}_{n \in \bb{N}^\op} \longrightarrow \{P_{n}F\}_{n \in \bb{N}^\op} \longrightarrow \{P_{\Vec{u}_n}F\}_{n \in \bb{N}^\op}
\]
which are compatible in the sense that for each $n \in \bb{N}$, the composites
\begin{align*}
    P_{mn}F &\longrightarrow P_{\Vec{v}_n}F \longrightarrow P_nF,\\
    P_{\Vec{v}_n}F &\longrightarrow P_nF \longrightarrow P_{\Vec{u}_n}F,
\end{align*}
gives the canonical map between excisive approximations, and that the limits of alternating rows are equivalent.
\end{lem}
\begin{proof}
To prove the compatibility claim we prove something slightly different which implies the claim. Let $\Vec{n}=(n_1, \dots,n_m)$ be a multi-index and denote by $N=N(\Vec{n}) = n_1+\cdots + n_m$. Since every $\Vec{n}$-excisive multi-variable functor is $N$-excisive as a functor of a single variable, there is a unique map $P_NF \to P_{\Vec{n}}F$ under $F$. Denote by $k = k(\Vec{n}) = \min\{n_i \mid 1\leq i \leq m\}$. By Lemma~\ref{lem: vec n to k}, there is a natural transformation $P_{\Vec{n}}F \to P_kF$ compatible with the maps from $F$. It suffices to show that the diagonal map
% https://q.uiver.app/#q=WzAsMyxbMCwwLCJQX3tcXFZlY3tufX1GIl0sWzEsMCwiUF9ORiJdLFswLDEsIlBfa0YiXSxbMCwyXSxbMSwwXSxbMSwyXV0=
\[\begin{tikzcd}
	{P_{\Vec{n}}F} & {P_NF} \\
	{P_kF}
	\arrow[from=1-1, to=2-1]
	\arrow[from=1-2, to=1-1]
	\arrow[from=1-2, to=2-1]
\end{tikzcd}\]
is the canonical map under $F$, but this follows from $P_kF$ being $N$-excisive and the universal property of $P_NF$. This proof works equally well by exchanging the roles of the multi-variable and single-variable functors, proving the other compatibility condition.

The statement about equivalent limits is another right cofinality argument, analogous to that of Lemma~\ref{lem: convergence on diagonal}, using that the ``multiplication by $m$'' functor
\[
m : \bb{N} \longrightarrow \bb{N}, n \longmapsto mn,
\]
is right cofinal, and similarly for the multi-variable limits.
\end{proof}

\begin{cor}\label{cor: diagonal and single variable convergence}
Let $F: \prod_{i=1}^m \es{C}_i \to \es{D}$, and let $n \in \bb{N}$. Denote by $\Vec{v}_n$ the multi-index $(n, \dots, n)$ and by $\Vec{u}_n$ the multi-index $(\lfloor \frac{n}{m} \rfloor, \dots, \lfloor \frac{n}{m} \rfloor)$. There are equivalences
\[
\lim_{n\in \bb{N}^\op}~P_{mn}F \to \lim_{n\in \bb{N}^\op}~P_{\Vec{v}_n}F \to \lim_{n\in \bb{N}^\op}~P_{n}F \to \lim_{n\in \bb{N}^\op}~P_{\Vec{u}_n}F.  
\]
\end{cor}

\begin{cor}\label{cor: single iff multi convergence}
Let $F: \prod_{i=1}^m \es{C}_i \to \es{D}$. The multi-variable Goodwillie calculus of $F$ converges at $\underline{X}$ if and only if the single variable Goodwillie tower of $F$ converges at $\underline{X}$.
\end{cor}
\begin{proof}
Combine Lemma~\ref{lem: convergence on diagonal} with Corollary~\ref{cor: diagonal and single variable convergence}.
\end{proof}

\subsection{Integral convergence of polyhedral products} Using the decompositions of Section~\ref{sec: integral decomps} the Goodwillie calculus of the polyhedral product functor $((-), \ast)^K$ converges after looping, hence the only obstruction to convergence is on the level of connected components.

\begin{lem}
Let $K$ be a simplicial complex on $[m]$ such that the fat wedge filtration on the real moment-angle complex on $K$ is trivial. If $\underline{X}$ is component-wise simply connected, then the polyhedral product functor 
\[
\Omega((-), \underline{\ast})^K : \es{S}_\ast^{\times m} \longrightarrow \T, \ \underline{X} \longmapsto \Omega(\underline{X}, \underline{\ast})^K,
\]
has convergent Goodwillie tower at $\underline{X}$ if and only if the polyhedral product functor 
\[
\Omega(C\Omega(-), \Omega(-))^K : \es{S}_\ast^{\times m} \longrightarrow \T, \ \underline{X} \longmapsto \Omega(C\Omega \underline{X}, \Omega\underline{X})^K,
\]
has convergent Goodwillie tower at $\underline{X}$.
\end{lem}
\begin{proof}
By Lemma~\ref{lem: decomposition of multi-variable excisive} it suffices to prove that the product functor has convergent Goodwillie tower at component-wise simply connected $\underline{X}$. This follows from Lemma~\ref{lem: multi-variable approx to product}, and the fact that the Goodwillie tower of the identity converges on simply connected spaces, see e.g.,~\cite[Theorem 1.13]{GoodCalcIII} and ~\cite[Example 4.3]{GoodCalcII}.
\end{proof}

We now show that under the conditions above, $\Omega(C(-), (-))^K$ converges integrally on simply connected spaces which come to us as suspensions of connected spaces. 

\begin{lem}
Let $K$ be a simplicial complex on $[m]$ such that the fat wedge filtration on the real moment-angle complex on $K$ is trivial. If $\underline{X}$ is component-wise connected, then the Goodwillie tower of the functor 
\[
\Omega(C\Omega(-), \Omega(-))^K \circ(\Sigma(-)) : \es{S}_\ast^{\times m} \longrightarrow \T, \ \underline{X} \longmapsto \Omega(C\Omega\Sigma\underline{X}, \Omega\Sigma\underline{X})^K,
\]
converges at $\underline{X}$.
\end{lem}
\begin{proof}
Let $n \in \bb{N}$ and denote by $\Vec{v}_n$ the multi-index $(n, \dots, n)$. By Lemma~\ref{lem: convergence on diagonal} it suffices to show that the natural map
\[
\Omega(C\Omega\Sigma\underline{X}, \Omega\Sigma\underline{X})^K \longrightarrow \lim_{n \in \bb{N}^\op}~P_{\Vec{v}_n}~(\Omega(C\Omega(-), \Omega(-))^K)(\Sigma\underline{X}),
\]
is an equivalence for component-wise connected $\underline{X}$.
By Theorem~\ref{thm: multivariable calc of fibre}, we have that
\[
\Omega P_{\Vec{v}_n} (C\Omega(-), \Omega(-))^K (\Sigma(\underline{X})) \simeq \Omega \sideset{}{'}\prod_{\omega \in \bb{L}(\mathscr{I}_m)} P_\kappa(\id)(\Sigma \omega \alpha (\underline{X})),
\]
where $\kappa = \min_i(\lfloor \frac{n}{a_i} \rfloor)$ for $a_i$ the smash power of $X_i$ appearing in the word $\omega(\alpha_{I_1,\Vec{k}_1}, \dots ,\alpha_{I_{a(\omega)},\Vec{k}_{a(\omega)}})$ when applied to $\underline{X}$.  
By Corollary~\ref{cor: weak product is product after Pn}, only finitely many terms of the weak product on the right are non-zero, hence there is no need to distinguish between the actual product and weak product so we have that 
\[
\Omega P_{\Vec{v}_n} (C\Omega(-), \Omega(-))^K (\Sigma(\underline{X})) \simeq \Omega \prod_{\omega \in \bb{L}(\mathscr{I}_m)} P_\kappa(\id)(\Sigma \omega \alpha (\underline{X})).
\]
Since limits commute we see that 
\[
\Omega \lim_{n \in \bb{N}^\op}P_{\Vec{n}} (C\Omega(-), \Omega(-))^K (\Sigma(\underline{X})) \simeq \Omega \prod_{\omega \in \bb{L}(\mathscr{I}_m)} \lim_{n \in \bb{N}^\op} P_\kappa(\id)(\Sigma \omega \alpha (\underline{X})),
\]
hence it suffices to prove that 
\[
\lim_{n \in \bb{N}^\op} P_\kappa(\id)(\Sigma \omega \alpha (\underline{X})) \simeq \lim_{n \in \bb{N}^\op} P_n(\id)(\Sigma \omega \alpha (\underline{X})),
\]
and invoke the convergence properties of the identity functor on simply connected spaces. This last equivalence is a left cofinality argument analogous to Corollary~\ref{cor: diagonal and single variable convergence}, noting that $\kappa = \min_i\left\lfloor \frac{n}{a_i} \right\rfloor = \left\lfloor \frac{n}{\max_{i}a_i} \right\rfloor$.
\end{proof}

The following corollary proves the first part of Theorem~\ref{Thm: convergence}.

\begin{cor}
Let $K$ be a simplicial complex on $[m]$ such that the fat wedge filtration on the real moment-angle complex on $K$ is trivial. If $\underline{X}$ is component-wise connected, then the Goodwillie tower of the functor 
\[
\Omega((-), \underline{\ast})^K\circ \Sigma : \es{S}_\ast^{\times m} \longrightarrow \T, \ \underline{X} \longmapsto \Omega(\Sigma\underline{X}, \underline{\ast})^K,
\]
converges at $\underline{X}$. \qed
\end{cor}

\begin{rem}
We would like to be able to conclude that the functor $((-), \underline{\ast})^K$ has a convergent Goodwillie tower. For this we would need an appropriate notion of \emph{analytic multi-variable} functors. Note that if $F: \T \to \T$ is analytic, see e.g.,~\cite[Definition 4.2]{GoodCalcII}, then the Goodwillie tower of $F$ converges since we can explicitly compute the connectivity of the maps $F(X) \to P_nF(X)$ as a linear function of $n$. In particular, if one knew the connectivity of the maps $\Omega F(X) \to \Omega P_nF(X) \simeq P_n\Omega F(X)$, then one could conclude the connectivity of $F(X) \to P_nF(X)$ which would still be linear in $n$ as looping introduces a constant. Unfortunately, checkable conditions for the convergence of multi-variable functors are a current gap in the literature and seem to be incredibly subtle.
\end{rem}

\subsection{Convergence after chromatic localization}
For details on unstable chromatic homotopy theory see \cite{HeutsHandbook}. Fix a prime $p$, and a height $h \geq 1$, and work $p$-locally throughout this section. Let $V$ denote a finite type $h$ complex with $v_h$ self-map $v: \Sigma^dV \to V$ for some natural number $d>0$. For a pointed space $X$, the $v$-periodic homotopy groups of a pointed space $X$ with coefficients in $V$, is
\[
v^{-1}\pi_\ast(X;V) = v^{-1}\pi_\ast(\Map_\ast(V,X)),
\]
given by inverting the action of $v$ on the homotopy groups of $\Map_\ast(V,X)$. The asymptotic uniqueness of $v_h$ self-maps implies that the $v$-periodic homotopy groups depend only on the height. Let $T(h)$ denote the mapping telescope of a $v_h$ self-map on a finite type $h$ spectrum. The Bousfield-Kuhn functor is a functor
\[
\Phi_h : \T \longrightarrow \s_{T(h)},
\]
with the property that the $v_h$-periodic homotopy groups of a space $X$ may be recovered from the homotopy groups of the Bousfield-Kuhn functor. A map $f: X \to Y$ is an equivalence in $v_h$-periodic homotopy if and only if it is sent to an equivalence of spectra by the Bousfield Kuhn functor. Note that a map $E \to F$ on $T(h)$-local spectra is an equivalence if and only if $E^V \to F^V$ is an equivalence for any finite type $h$ complex $V$ with a $v_h$ self-map $v : \Sigma^dV \to V$, hence $T(h)$ may be described as the mapping telescope of $v$, and hence
\[
\pi_\ast(\Phi_h(X)^V) \cong v^{-1}\pi_\ast(X;V). 
\]
We show that the Goodwillie tower of the polyhedral product functor $(-, \underline{\ast})^K$ diverges in $v_h$-periodic homotopy under our continuing mild hypotheses on $K$. Our arguments here are heavily based on arguments of Brantner and Heuts~\cite[\S3]{BrantnerHeuts} in the case of a wedge, i.e., when $K$ is a finite collection of points and where our Ganea style decomposition is replaced by the Hilton-Milnor splitting. 

By~\cite[Theorem 1.1]{Kuhn}, there is an equivalence,
\[
\Phi_h(X)^V \simeq \Phi_v(X),
\]
where $\Phi_v(X)$ is the spectrum defined by 
\[
\Phi_v(X)_{kd} = \Map_\ast(V,X),
\]
with structure maps given by precomposition with the $v_h$ self-map $v: \Sigma^dV \to V$. We say that the Goodwillie calculus of $F: \es{S}^{\times m}_\ast \to T$ \emph{diverges in $v_h$-periodic homotopy at $\underline{X}$} if the map
\[
\Phi_vF(\underline{X}) \longrightarrow \lim_{n \in \bb{N}^\op} \Phi_v~ P_nF(\underline{X}),
\]
is not an equivalence. Note that by Corollary~\ref{cor: single iff multi convergence} for convergence questions there is no difference between single-variable and multi-variable calculus. 

\begin{lem}\label{lem: vh divergence of multi-calc}
Let $K$ be a simplicial complex on $[m]$ such that the fat wedge filtration on the real moment-angle complex on $K$ is trivial. If $\underline{X}$ is component-wise connected and $\Phi_v(\Sigma \omega \alpha (\underline{X}))$ is not contractible for infinitely many $\omega \in \bb{L}(\mathscr{I}_m)$, then the canonical map 
\[
\Phi_v(\Omega(C\Omega\Sigma\underline{X}, \Omega\Sigma\underline{X})^K) \longrightarrow  \lim_{n \in \bb{N}^\op} \Phi_v P_{\vec{v}_n}(\Omega(C\Omega(-), \Omega(-))^K)(\Sigma\underline{X}),
\]
is not an equivalence.
\end{lem}
\begin{proof}
The argument is all but identical to~\cite[Corollary 3.3]{BrantnerHeuts}, and replies to two key facts: firstly, that the weak product can be replaced by the actual product, and secondly, that the map
\[
\Phi_v(\Omega(C\Omega\Sigma\underline{X}, \Omega\Sigma\underline{X})^K) \longrightarrow \lim_{n \in \bb{N}^\op} \Phi_v P_{\vec{v}_n}(\Omega(C\Omega(-), \Omega(-))^K)(\Sigma\underline{X}),
\]
may be identified up to equivalence with the evident map
\[
 \sideset{}{'}\prod_{\omega \in \bb{L}(\mathscr{I}_m)} \Phi_v\Sigma \omega \alpha (\underline{X}) \longrightarrow  \prod_{\omega \in \bb{L}(\mathscr{I}_m)} \lim_{n} \Phi_vP_n(\id)(\Sigma \omega \alpha  (\underline{X})).
\]
The proof is then a diagram chase on homotopy groups using the periodicity of the homotopy groups of $\Phi_v$ and the pigeonhole principle.
\end{proof}

The following theorem completes the proof of Theorem~\ref{Thm: convergence}.

\begin{thm}
Let $K \neq \Delta^m$ be a simplicial complex on $[m]$ such that the fat wedge filtration on the real moment-angle complex on $K$ is trivial. The $v_h$-periodic Goodwillie calculus fails to converge on spheres of dimension at least two.
\end{thm}
\begin{proof}
By Corollary~\ref{cor: single iff multi convergence} it suffices to show that the multi-variable calculus does not converge on spheres of dimension at least two. Let $\underline{X}$ be a sequence of spheres of dimension at least one (so that their suspensions have dimension at least two). By Lemma~\ref{lem: vh divergence of multi-calc}, we must establish that there are infinitely many $\omega \in \bb{L}(\mathscr{I}_m)$ for which $\Phi_v(\Sigma \omega \alpha (\underline{X}))$ is not contractible.

Since $K$ is not the full simplex, it has at least one minimal missing face, $I_0$, and for such a face we have $K_{I_0} = \partial \Delta^{I_0}$, so $|K_{I_0}|$ is an $(|I_0|-1)$-sphere. Since $\underline{X}$ consists of spheres, we have that each $\alpha_{I_0,\vec{k}}(\underline{X})$ is a sphere of dimension at least two, and since $|I_0| \geq 2$ there are countably many such $\vec{k}$. The Lie basis $\bb{L}(\mathscr{I}_m)$ contains all of the lie words $\omega$ on these countably many generators, and each of these words yields a $\Sigma \omega \alpha (\underline{X})$ which is a sphere of dimension at least two, hence a $\Phi_v(\Sigma \omega \alpha (\underline{X}))$ which is not contractible, as required.
\end{proof}

\bibliographystyle{alpha}
\bibliography{references} 
\end{document}